\documentclass[reqno,a4paper,11pt]{amsart}

\usepackage{mystyle}
\usepackage{etoolbox}

\usepackage[headings]{fullpage}
\usepackage{parskip}

\usepackage[utf8]{inputenc}
\usepackage[T1]{fontenc}
\usepackage[british]{babel}
\usepackage{microtype,dsfont}

\usepackage[dvipsnames]{xcolor}
\usepackage[pdfencoding=auto,pdfusetitle]{hyperref}
\hypersetup{colorlinks=true, linkcolor=blue, citecolor=teal}

\usepackage{amsfonts,amssymb,bbm,mathrsfs}
\usepackage{amsmath}
\usepackage{mathtools}
\usepackage{centernot,url,booktabs}
\usepackage[noabbrev,capitalize]{cleveref}
\usepackage[shortlabels]{enumitem}
\usepackage{float}

\usepackage{tikz}
\usetikzlibrary{matrix}
\usetikzlibrary{positioning}
\usetikzlibrary{calc}
\usetikzlibrary{arrows}
\tikzset{> =stealth}

\usepackage{tikz-cd}

\newcommand{\addQEDstyle}[2]{\AtBeginEnvironment{#1}{\pushQED{\qed}\renewcommand{\qedsymbol}{#2}}\AtEndEnvironment{#1}{\popQED}}

\theoremstyle{plain}
\newtheorem{theorem}{Theorem}[section]
\newtheorem{lemma}[theorem]{Lemma}
\newtheorem{proposition}[theorem]{Proposition}
\newtheorem{corollary}[theorem]{Corollary}

\theoremstyle{definition}
\newtheorem{definition}[theorem]{Definition}

\newtheorem{example}[theorem]{Example}
\addQEDstyle{example}{$\triangle$}

\newtheorem{construction}[theorem]{Construction}
\addQEDstyle{construction}{$\triangle$}

\theoremstyle{remark}
\newtheorem{remark}[theorem]{Remark}
\addQEDstyle{remark}{$\triangle$}

\renewcommand{\epsilon}{\varepsilon}
\renewcommand{\phi}{\varphi}

\newcommand{\N}{\mathds{N}}
\newcommand{\Z}{\mathds{Z}}
 
\newcommand{\R}{\mathds{R}}
\newcommand{\C}{\mathds{C}}

\newcommand{\da}{{\downarrow}}

\newcommand{\Ring}{\mathbf{Ring}}
\newcommand{\CRing}{\mathbf{CRing}}
\newcommand{\Rig}{\mathbf{Rig}}
\newcommand{\CRig}{\mathbf{CRig}}
\newcommand{\InvRig}{\mathbf{InvRig}}
\newcommand{\InvCRig}{\mathbf{InvCRig}}
\newcommand{\IdemRig}{\mathbf{IdemRig}}

\newcommand{\Ker}{\mathrm{Ker}}

\newcommand{\End}{\mathrm{End}}

\mathchardef\mhyphen="2D

\newcommand{\G}{\mathcal{G}}
\newcommand{\I}{\mathcal{I}}
\newcommand{\z}{\mathfrak{z}}
\newcommand{\h}{\mathfrak{h}}

\title{The Case for Inverse Semirings}

\author[P. F. Faul]{Peter F. Faul}
\address{Stellenbosch University, Stellenbosch, South Africa}
\email{peter@faul.io}

\author[A. Goswami]{Amartya Goswami}
\address{University of Johannesburg, Auckland Park Kingsway Campus, Auckland Park 2006, South Africa}
\email{agoswami@uj.ac.za}

\author[G. Joubert]{Gideo Joubert}
\address{Stellenbosch University, Stellenbosch, South Africa}
\email{22930701@sun.ac.za}

\author[G. Manuell]{Graham Manuell}
\address{Stellenbosch University, Stellenbosch, South Africa}
\email{graham@manuell.me}

\date{18 November 2024}


\begin{document}

\maketitle
\thispagestyle{empty}

\begin{abstract}
A semiring generalises the notion of a ring, replacing the additive abelian group structure with that of a commutative monoid. In this paper, we study a notion positioned between a ring and a semiring --- a semiring whose additive monoid is a commutative inverse semigroup. These inverse semirings include some important classes of semiring, as well as some new motivating examples. We devote particular attention to the inverse semiring of bounded polynomials and argue for their computational significance. We then prove a number of fundamental results about inverse semirings, their modules and their ideals. Parts of the theory show strong similarities with rings, while other parts are akin to the theory of idempotent semirings or distributive lattices. We note in particular that downward-closed submodules are precisely kernels. We end by exploring a connection to the theory of E-unitary inverse semigroups.
\end{abstract}

\section{Introduction}

Rings and distributive lattices share a number of properties and both fall under the umbrella of \emph{semirings} --- a very general algebraic structure in which the additive group structure of a ring has been replaced by a commutative monoid. However, both of these examples already occur in a more restrictive class of semiring, which has a weak notion of `additive inverse' in the following sense. 

\begin{definition}
An \emph{inverse monoid} is a monoid $(M,+,0)$ for which each $x \in M$ has a unique `inverse' element $-x$ such that $x + (-x) + x = x$ and $(-x) + x + (-x) = -x$.
\end{definition}

\begin{definition}
An \emph{inverse semiring} is a semiring $(R,+,0, \cdot, 1)$ in which the additive monoid $(R,+,0)$ is an inverse monoid.
\end{definition}

Inverse semirings have appeared before in the literature under various names. They were introduced in \cite{Kar74}, where they were called additively inversive semirings. They have also been called inverse semiring \cite{AD21}, additively inverse semiring \cite{dadhwal2024study}, additive inverse semiring \cite{SM04}, and additively regular semiring \cite{Kar74}. Pure ideals and pure $k$-ideals of inverse semirings have been studied in \cite{palakawong2023characterizations}, semiprime inverse semirings in  \cite{sara2016centralizer,dog2021centralizers} and the special class of so-called Clifford semirings in \cite{sen2005clifford, bhuniya2014clifford}. Derivations of various types in inverse semirings are studied in \cite{yilmaz2023note, dadhwal2024study, yaqoub2021generalized}. Despite this rather extensive list of papers, the basic theory and motivation for inverse semirings does not seem to have been written down before.

Every ring and every idempotent semiring is an example of an inverse semiring. In the former case, the inverse is given by the usual negation, while in the latter it is the identity. An example that does not lie in either of these classes is given by the endomorphism semiring of a commutative inverse monoid.
Another important example is the inverse semiring of `bounded polynomials', which we claim captures the computational nature of polynomials more accurately than the usual ring of polynomials (see \cref{ex:polynomials}).

The theory of inverse semirings has much in common with that of rings and idempotent semirings. In particular, they come equipped with a natural order structure, and ideals which are downward-closed with respect to this order behave very much like ideals of rings. More precisely, these ideals are the \emph{subtractive} ideals and appear as the kernels of homomorphisms of inverse semirings. Similar results are true of modules over an inverse semiring. On the other hand, an ideal is upward-closed exactly when it contains the set of additive idempotents.

Rings and idempotent semirings are not the only possible sources of inspiration. In fact, many concepts and results concerning inverse semigroups naturally extend to inverse semirings. For instance, the set of additive idempotents of an inverse semiring will be a subtractive ideal precisely when the inverse semiring is E-unitary. Moreover, every E-unitary inverse semiring embeds into the product of a ring and an idempotent semiring (see \cref{thm:e_unitary_embedding}).

In the study of inverse semirings, we occasionally encounter some near misses --- that is, algebraic structures that are almost inverse semirings. In particular, we encounter a number of structures that are inverse semirings except that $0 \cdot x \ne 0$ in general. The fact that these examples have an additive identity is perhaps accidental, and it may be best to think of their additive structure as \emph{inverse semigroups} rather than inverse monoids. The Riemann sphere from complex analysis constitutes one such example (see \cref{ex:reimann_sphere}).

\section{Background}

\subsection{Inverse monoids}

There are a number of basic results concerning (commutative) inverse monoids that we intend to take for granted in this paper. For a proper introduction to this topic see \cite{lawson1998inverse}. 

\begin{proposition}
In an inverse monoid, $x-x$ is an idempotent element and we write $0_x = x-x$ for the idempotent associated to $x$ in this way.
\end{proposition}

\begin{proposition}\label{prp:sum_of_idempotents}
In an inverse monoid, $0_{x+y} = 0_x + 0_y$.
\end{proposition}

\begin{proposition}\label{prop:commuting_idemponents}
Let $M$ be a monoid such that for every $x \in M$ there is a $y \in M$ with $xyx = x$. Then $M$ is an inverse monoid if and only if the idempotents in $M$ commute. Consequently, if we consider taking inverses to be a unary operation, we see that inverse monoids form a variety.
\end{proposition}

\begin{remark}\label{rem:commuting_idemponents}
The above proposition implies that we do not need to prove the uniqueness of inverses if the monoid is commutative.
\end{remark}

\begin{proposition}
Let $M$ be an inverse monoid. The set $E(M)$ of idempotents of $M$ forms a semilattice under addition.
\end{proposition}

\begin{proposition}
Let $M$ be a commutative inverse monoid and suppose $z \in E(M)$. The set $\G(z) = \{s \mid 0_s = z\}$ is an abelian group with identity $z$.
\end{proposition}

It is possible to define an order of the elements of an inverse monoid. We consider the opposite of the traditional order in this paper since we consider the monoid operation to be additive instead of multiplicative, i.e., we view $E(M)$ as a $\vee$-semilattice.

\begin{definition}\label{InvSR preorder}
Let $M$ be an inverse monoid. Then we define the canonical preorder on $M$ by $x \le y$ if there exists an idempotent $z$ such that $x + z = y$.
\end{definition}

\begin{proposition}\label{prp:preorderadd}
In an inverse monoid, $x \le y$ if and only if $x + 0_y = y$.
\end{proposition}

There are two main results concerning $0$ and the order. Firstly, nothing compares less than zero and secondly, only idempotents compare greater than $0$.

\begin{proposition}
In an inverse monoid, $x \le 0$ if and only if $x = 0$.
\end{proposition}

\begin{proposition}\label{prp:upclosed}
In an inverse monoid, $0 \le x$ if and only if $x$ is an idempotent.
\end{proposition}

\subsection{Semirings}

We also make use of some basic concepts from the theory of semirings. See \cite{Gol99} for more details.

\begin{definition}
A \emph{semiring} $(R,+,0, \times ,1)$ has an `additive' commutative monoid $(R,+,0)$ and a `multiplicative' monoid $(R,\times,1)$ such that multiplication distributes over addition. We write $\Rig$ for the category of semirings, where the morphisms are the expected notion of semiring homomorphisms.

We say a semiring is \emph{commutative} if its multiplicative monoid is commutative and write $\CRig$ for the corresponding category.
\end{definition}

The category $\Ring$ of rings and $\CRing$ of commutative rings are subcategories of $\Rig$ and $\CRig$, respectively. Another important class of semirings is that of idempotent semirings.

\begin{definition}
An \emph{idempotent semiring} $(R,+,\times, 0,1)$ is a semiring in which $+$ is idempotent --- that is, $x + x = x$ for all $x$. We denote the category of idempotent semirings by $\IdemRig$.
\end{definition}

Finally, we write $\InvRig$ for the category of inverse semirings and $\InvCRig$ for the subcategory of commutative structures.

In this paper we will be interested in modules and algebras over inverse semirings.

\begin{definition}
Given a semiring $R$, a commutative monoid $(M,+,0)$ becomes a (left) $R$-module by equipping it with an action $R\times M\to M$, satisfying that for all $a,b\in R$ and $x,y\in M$ we have
\begin{enumerate}[(i)]
\item $a(x+y)=ax + ay$,
\item $a0=0$,
\item $(a+b)x=ax + bx$,
\item $0x=0$,
\item $(ab)x=a(bx)$,
\item $1x=x$.
\end{enumerate}
\end{definition}

\begin{definition}
If $R$ is a commutative semiring, then an \emph{$R$-algebra} $A$ is an $R$-module equipped with a compatible semiring structure. Explicitly, it is a module $A$ with an unital associative bilinear multiplication such that $r\cdot (a b) = (r\cdot a) b = a (r \cdot b)$, where $r$ is a scalar in $R$, $\cdot$ denotes scalar multiplication, and the algebra multiplication is denoted by juxtaposition.
\end{definition}
Note that a commutative $R$-algebra $A$ is the same thing as an inverse semiring $A$ equipped with a semiring homomorphism from $R$ to $A$.

\section{Examples}

Just as an inverse semigroup simultaneously generalises groups and semilattices, inverse semirings simultaneously generalise rings and idempotent semirings, the latter of which include distributive lattices and the other following examples.

\begin{example}\label{ex:quantale}
A \emph{quantale} $(Q,\bigvee,0,\cdot,1)$ is a complete join-semilattice equipped with a distributive multiplication. Since join is idempotent, quantales are idempotent semirings.
\end{example}

\begin{example}
An important example is given by the \emph{tropical semiring} $(\mathds{R}_\infty,\min,\infty,+,0)$. Here $\mathds{R}_\infty = \R \sqcup \{\infty\}$. Do note that, somewhat confusingly, $+$ is the \emph{multiplication} and $\min$ is the idempotent addition.
\end{example}

There are a number of interesting examples that are neither rings nor idempotent semirings.

\begin{example}\label{endInvSR}
Let $X$ be a commutative inverse monoid. Then the set of endomorphisms $\End(X)$ has the structure of an inverse semiring. Addition is computed pointwise and multiplication is given by function composition.
\end{example}

If the inverse monoid $X$ is neither a group nor a semilattice, this inverse semiring will be neither a ring nor an idempotent semiring. All inverse semirings $R$ embed into $\End(R)$ by a version of Cayley's theorem for semirings.

Another fundamental example stems from computational considerations about polynomials.
A polynomial is typically represented on a computer by a finite list of coefficients. The length of this list gives an upper bound on the degree of the polynomial, but since the coefficients of the highest powers of $x$ might be zero, this bound is not necessarily tight.

Note that in the case of real or complex coefficients, it is \emph{not} possible to ensure that this bound is tight, since checking whether a real number is zero is undecidable. (Alternatively, checking if a floating-point number is zero is undesirable for reasons of numerical stability.)
For example, consider two polynomials $p(x) = x^2 + x$ and $q(x) = -x^2$, both with degree bound $2$. Term-by-term addition finds $p(x) + q(x)$ to be $0x^2 + x$. This polynomial has degree 1, but the computer cannot determine this and so cannot reduce the size of the array.
Thus, the usual ring of polynomials is not actually a faithful model of this situation. A better model is given by the following inverse semiring.

\begin{example} \label{ex:polynomials}
Let $\R[x]_B$ be the set of pairs $(p,n)$, where $p \in \R[x]$ and $n \in \N \sqcup \{-\infty\}$ such that $\mathrm{deg}(p) \le n$ (where $0$ is defined to have degree = $-\infty$).
This has an inverse semiring structure given by
\begin{enumerate}[(i)]
\item $(p,n) + (q,m) = (p+q, \max(n,m))$,

\item $0 = (0,-\infty)$,

\item $(p,n)(q,m) = (pq,n+m)$,

\item $1 = (1,0)$,

\item $-(p,n) = (-p,n)$.
\end{enumerate}
Note that the additive identity is not $(0,0)$ because then multiplication by $0$ would not behave appropriately, but see \cref{rem:zeroless_polys}.
\end{example}

A large class of examples can be obtained from rings and bimodules by the following construction.
\begin{construction}\label{constr:from_bimodule}
Let $S$ be a ring, $A$ an $S$-$S$-bimodule, and $f\colon A \to S$ a bimodule homomorphism (where $S$ is equipped with its canonical bimodule structure). Suppose that $f(x) \cdot y = x \cdot f(y)$ for all $x,y \in A$. Then there is an inverse semiring $\I(f)$ with underlying set $S \sqcup A$ and operations defined as follows:
\begin{enumerate}[(i)]
\item binary sums/products of elements in the $S$ component are computed as in $S$,
\item the multiplicative unit is given by $1_S \in S$,
\item binary sums of elements in the $A$ component are computed as in $A$,
\item the additive unit is given by $0_A \in A$,
\item $a a' = f(a) \cdot a' = a \cdot f(a')$ for $a,a' \in A$,
\item $ s + a = a + s = f(a) + s$ for $a \in A$ and $s \in S$,
\item $a s = a \cdot s$,
\item $s a = s \cdot a$.
\end{enumerate}
Note in particular that these conditions hold if $A$ is a two-sided ideal of $S$ and $f$ is the inclusion.
\end{construction}

This is indeed an inverse semiring. Moreover, the only idempotents are $0 = 0_A$ and $0_1 = 0_S$.
\begin{proof}
It is clear that addition is commutative and that $0_A$ is an additive unit. Associativity is also straightforward by checking each case in turn and using the linearity of $f$.

It is easy to see that $1_S$ is a multiplicative unit. For associativity of multiplication, $(xy)z = x(yz)$, there are 8 cases. If $x,y,z \in S$, this follows from the associativity of multiplication in $S$. If $x,y,z \in A$, this follows from the linearity of $f$ and the compatibility of scalar multiplication with multiplication in $S$. If two of the three elements are in $S$, it follows from the compatibility of scalar multiplication with multiplication in $S$ and of left and right scalar multiplications with each other. The latter compatibility also gives associativity for the remaining three cases.
    
Distributivity ($x(y+z) = xy + xz$ and its dual) also has a number of cases, most of which follow from various distributivity axioms in the data. If $y \in A$ and $z \in S$, we also use the linearity of $f$.
The absorption law $0x = 0 = x0$ follows from absorption for $S$ and the fact that $f$ preserves $0$.
    
Finally, additive inverses are quickly seen to be given by the negatives in $S$ and $A$.
\end{proof}
Later in \cref{cor:two_idempotents_from_bimodules} we will see that this construction exhausts all inverse semirings with two (additive) idempotents. In fact, this is a special of a much more general construction that will be discussed in \cite{faul2024}.

We can also obtain a surprisingly useful class of inverse semirings by adjoining zeros to rings.
\begin{definition}\label{def:adjoin_zero}
The functor $(-)_0 \colon \InvRig \to \InvRig$ sends an inverse semiring $R$ to the inverse semiring $R_0 = R \sqcup 0$, where the addition, multiplication and inverses are defined on $R$ as before, while the new $0$ is an identity for addition and an absorbing element with respect to multiplication.
As for morphisms, $f\colon R \to S$ is sent to a map $f_0\colon R_0 \to S_0$ which acts on the old elements like $f$ does and sends the new $0$ to $0$.
We will denote the zero of $R$ in $R_0$ by $\widehat{0}$ to distinguish it from the newly adjoined zero.
\end{definition}
Note that, at least when $R$ is a ring, $R_0$ can be obtained from a special case of \cref{constr:from_bimodule} with $A = 0$.

\begin{example}
In particular, applying this construction to the rings $\Z$ and $\R$ yields the inverse semirings $\Z_0$ and $\R_0$, which we will see more of later in the paper.
\end{example}

\subsection{Near misses}

There are a number of algebraic structures that are almost, but not quite, inverse semirings. In particular, these examples are \emph{non-absorptive} in the sense that $0 \cdot x$ need not equal $0$. In fact, in all the natural examples we are aware of, $0 \cdot x = 0_x$.

It is perhaps best to think of this as a variant of inverse semirings in which the additive structure might lack a zero --- that is, it forms only an inverse semigroup under addition, instead of an inverse monoid. We call these \emph{zeroless inverse semirings}. While a zero element does exist in all the examples we consider, it will be useful to not require it to be preserved by homomorphisms.

The functor $(-)_0$ from \cref{def:adjoin_zero} can be understood as arising from the adjunction between zeroless inverse semirings and inverse semirings. The forgetful functor from inverse semirings to zeroless inverse semirings has a left adjoint which freely adds a zero.

Having discussed adjoining $0$, we now consider a notion of adjoining $\infty$. 

\begin{definition}
Given a ring $R$, let $R_\infty$ be $R$ with an additional $\infty$ element added, where the addition, multiplication and inverses are defined on elements of $R$ as before, and where $x + \infty = \infty$ and $x \cdot \infty = \infty = \infty \cdot x$ for all $x \in R_\infty$.
\end{definition}

\begin{proposition}
If $R$ is a ring, $R_\infty$ is a zeroless inverse semiring.
\end{proposition}

We now have the following example.

\begin{example}\label{ex:reimann_sphere}
The Riemann sphere $\C_\infty$, with the natural addition and multiplication but with $0 \cdot \infty = \infty$ is a non-absorptive inverse semiring. Of course, we can obtain a proper inverse semiring by adjoining $0$.
\end{example}

We will see some more examples of zeroless semirings later in the paper.

\section{Fundamental results}\label{sec:fundamental}

In this section, we establish some basic facts about inverse semirings. We begin by looking at the interaction between the additive inverse monoid and the multiplicative monoid. We then study the element $0_1$ in some detail.

The first result establishes that multiplication by an idempotent is absorptive.

\begin{proposition}\label{prp:absorb}
Let $R$ be an inverse semiring and let $x,y \in R$. Then $0_x \cdot y = 0_{xy}$ and $y \cdot 0_x = 0_{yx}$.
\end{proposition}

\begin{proof}
Simply observe $0_{xy} = xy-xy = (x-x)y = 0_xy$. The other equation is dual.
\end{proof}
This means, in particular, that one can obtain the idempotent associated to any element by simply multiplying it by the distinguished element $0_1$.

It also follows from \cref{prp:absorb} that just as the idempotent elements of an inverse monoid form a semilattice, the additive idempotents of an inverse semiring $R$ form an idempotent semiring $E(R)$, where $0_1$ is the multiplicative identity.

\begin{corollary}\label{cor:idempotent_of_product}
Let $R$ be an inverse semiring and let $x,y \in R$. Then $0_x \cdot 0_y = 0_{xy}$. Hence $E(R)$ is an idempotent semiring.
\end{corollary}
\begin{proof}
We have $0_x0_y = 0_{x0_y} = 0_{0_{xy}} = 0_{xy}$, as required.
\end{proof}
Note that $E(R)$ is generally not a sub-semiring of $R$, since $1$ is seldom an idempotent. It is, however, a two-sided ideal by the above proposition.

Next, we establish that multiplication by $-1$ acts as one might hope.

\begin{proposition}
Let $R$ be an inverse semiring and let $x \in R$. Then $-1 \cdot x = -x = x \cdot -1$.
\end{proposition}

\begin{proof}
Consider $x + -1\cdot x = (1 - 1)x = 0_1 x = 0_x$. A symmetric version of this argument gives $-1 \cdot x + x = 0_x$. The result follows by the uniqueness of inverses.
\end{proof}

Note that multiplication respects the order.

\begin{proposition}\label{ledotmonotone}
Let $R$ be an inverse semiring and let $x \le y$ and $u \le v$. Then $xu \le yv$.
\end{proposition}

\begin{proof}
By assumption, we can write $y = x + z_1$ as $v = u + z_2$ for additive idempotents $z_1$ and $z_2$. Therefore, we have $yv = xu + xz_2 + z_1u + z_1z_2$. But $xz_2$, $z_1u$ and $z_1z_2$ are idempotents by \cref{prp:absorb} and so $yv \le xu$ as required.
\end{proof}

The element $0_1$ controls the behaviour of an inverse semiring in the following sense.

\begin{proposition}\label{prp:zeroone}
Let $R$ be an inverse semiring. Then
\begin{enumerate}[(i)]
\item $R$ is a ring if and only if $0_1 = 0$,
\item $R$ is an idempotent semiring if and only if $0_1 = 1$.
\end{enumerate}
\end{proposition}

\begin{proof}
The forward implications are clear. If $0_1 = 0$ then for all $x \in R$ we have $0_x = 0_1\cdot x = 0 \cdot x = 0$, and so $x$ is (additively) invertible.
On the other hand, if $0_1 = 1$, then $0_x = 0_1\cdot x= 1\cdot x = x$, and so $x$ is (additively) idempotent.
\end{proof}

We now take a look at the groups associated to each idempotent.

\begin{proposition}
Let $R$ be an inverse semiring. Then $\G(0)$ is a two-sided ideal.
\end{proposition}

\begin{proof}
Note that $\G(0)$ contains $0$ and is closed under addition. To check that multiplication is absorptive take $x \in \G(0)$ and $y \in R$. We must show that $0_{xy} = 0$. But $0_{xy} = 0_x0_y = 0 \times 0_y = 0$. Multiplication on the other side is similar.
\end{proof}

\begin{proposition}
Let $R$ be an inverse semiring. Then $\G(0_1)$ is a ring.
\end{proposition}

\begin{proof}
We have that $\G(0_1)$ is an abelian group and contains the multiplicative unit $1$. We need only check that it is closed under multiplication. Let $x,y \in \G(0_1)$. We must show that $0_{xy} = 0_{1}$. But $0_{xy} = 0_x0_y = 0_10_1 = 0_1$.
\end{proof}

We call $\G(0_1)$ the \emph{ring of scalars} of $R$. Multiplication by elements of $\G(0_1)$ will scale an element without changing the associated idempotent. 
This name is justified by the following proposition.

\begin{proposition}\label{prop:bimodule_over_scalars}
Let $R$ be an inverse semiring and $z$ an additive idempotent. Then $\G(z)$ is a $\G(0_1)$-bimodule where for $s \in \G(0_1)$ and $s \in \G(z)$, we have $s \cdot x = sx$ and $x \cdot s = xs$.
\end{proposition}

\begin{proof}
Recall that $G(z)$ is an abelian group. The module operation is simply the inverse semiring multiplication. Hence we need only check that if $s \in \G(0_1)$ and $x \in \G(z)$ then $s \cdot x \in \G(z)$ and that $x \cdot s \in \G(z)$. For the first case, consider $0_{sx} = 0_s 0_x = 0_1 z = z$, and so $s\cdot x \in \G(z)$ as required. A symmetric argument gives the dual result.
\end{proof}

\begin{example}
Let $\R[x]_B$ be the inverse semiring of bounded polynomials from \cref{ex:polynomials}. Then $0_1 = (0,0)$ and $\G(0,0)$ consists of the elements $(r,0)$ where $r \in \R$. Hence, $\G(0_1)$ is precisely the field of scalars $\R$. A similar result holds if $\R$ is replaced with any other ring.
\end{example}

The following proposition shows that $\G(0)$ and $\G(0_1)$ always satisfy the conditions of \cref{constr:from_bimodule}.

\begin{proposition}
Let $R$ be an inverse semiring. The map $\z\colon \G(0) \to \G(0_1)$ with $\z(x) = x + 0_1$ is a $\G(0_1)$-bimodule homomorphism satisfying $\z(x) \cdot y = x \cdot \z(y)$.
\end{proposition}

\begin{proof}
We begin by showing that $\z$ is a bimodule homomorphism. Let $x \in \G(0)$ and $s \in \G(0_1)$. Then $s \cdot \z(x) = s \cdot (x + 0_1) = s\cdot x + s \cdot 0_1 = s \cdot x + 0_1 = \z(s \cdot x)$. Multiplication on the right-hand side is proved similarly. Finally, if $x,y \in \G(0)$ then
\[\z(x) \cdot y = (x + 0_1) \cdot y = xy + 0_1y = xy + 0 = xy + x 0_1 = x \cdot (y + 0_1) = x \cdot \z(y)\]
and so the result follows.
\end{proof}

Now applying \cref{constr:from_bimodule}, we arrive at the following proposition.
\begin{proposition}
Let $R$ be an inverse semiring. Then there is an induced homomorphism $\h\colon \I(\z\colon \G(0) \to \G(0_1)) \to R$ sending $s \in \G(0_1)$ to $s \in R$ and $x \in \G(0)$ to $x \in R$. Moreover, if $R$ is not a ring, this homomorphism is injective. We call $\I(\z)$ the \emph{heart} of $R$.
\end{proposition}

\begin{proof}
The map $\h$ clearly preserves $0$ and $1$. To see that it preserves binary addition, the only non-trivial case is where $x \in \G(0)$ and $y \in \G(0_1)$. Then in $\I(\z)$ we have $x + y = \z(x) + y$, and so $\h(x + y) = \z(x) + y = x + 0_1 + y = x + y$, since $0_y = 0_1$. On the other hand, $\h(x) + \h(y) = x + y$, as expected.

For multiplication, there are three non-trivial cases. If $x, y \in G(0)$, then $\h(xy) = \h(\z(x) \cdot y) = (x + 0_1)y = xy + 0 = xy = h(x)h(y)$. If $x \in \G(0)$ and $y \in \G(0_1)$ then $\h(xy) = x\cdot y = xy = h(x)h(y)$, as required. The dual case is similar. Thus, $\h$ is a semiring homomorphism.

Finally, if $R$ is not a ring, then $0_1 \ne 0$ and so $\G(0)$ and $\G(0_1)$ are disjoint. Thus, the $\h$ is clearly injective.
\end{proof}

\begin{corollary}\label{cor:two_idempotents_from_bimodules}
\Cref{constr:from_bimodule} gives rise to \emph{every} inverse semiring with exactly two idempotents.
\end{corollary}

\Cref{prp:zeroone} can be used to describe two relevant adjunctions.
As discussed before, both rings and idempotent semirings are examples of inverse semirings and we can view these relationships as subcategory inclusions.
\begin{proposition}\label{prop:reflections}
Both the inclusion $\Ring \hookrightarrow \InvRig$ and the inclusion $\IdemRig \hookrightarrow \InvRig$ have left adjoints given by quotienting by the congruence generated by $0_1 \sim 0$ and $0_1 \sim 1$, respectively. In the latter case, the reflection of an inverse semiring $R$ is isomorphic to $E(R)$.
\end{proposition}

\begin{proof}
 Adding equations to a variety always yields a reflective subcategory whose unit takes the largest quotient making the equations hold. A ring is an inverse semiring $R$ such that $x - x = 0$ for all $x \in R$, and an idempotent semiring is an inverse semiring such that $x+x = x$. By \cref{prp:zeroone} it is enough to ask that $0_1 = 0$ and $0_1 = 1$, respectively.

 For the reflection to idempotent semirings, the condition can equivalently be described as setting $x = 0_x$ for all $x$. Note that for an inverse semiring $R$, the map from $R$ to $E(R)$ sending $x$ to $0_x$ is a semiring homomorphism by \cref{prp:sum_of_idempotents} and \cref{cor:idempotent_of_product}. It is clearly surjective, but is injective on idempotents. Thus, it makes precisely the identification made by the idempotent reflection.
\end{proof}

\begin{remark}
One might also ask about \emph{right} adjoints to the above inclusions. Strictly speaking, these do not exist, but there is a sense in which they almost do. For example, the left adjoint $E$ of the inclusion $\IdemRig \hookrightarrow \InvRig$ would also be the right adjoint if we were working with \emph{non-unital} (semi)rings. The issue is that the candidate counit $E(R) \hookrightarrow R$ fails to be a semiring homomorphism because it does not send $0_1$ to $1$.

The case of $\Ring \hookrightarrow \InvRig$ is, perhaps, even more interesting. If we work with non-unital structures then $G(0)$ provides the right adjoint with the counit given by the inclusion $G(0) \hookrightarrow R$. If we instead work with zeroless semirings, then $G(0_1)$ gives the right adjoint, with counit $G(0_1) \hookrightarrow R$.
\end{remark}

\section{Algebras and free structures}

Polynomial rings can be understood as free commutative algebras over a ring. In this section, we will discuss the analogous case of free algebras over an inverse semiring.

The simplest case to consider is the free inverse semiring on the empty set --- that is, the initial inverse semiring.

\begin{proposition}
The initial inverse semiring is $\Z_0$ --- the ring of integers $\Z$ with a new zero adjoined (see \cref{def:adjoin_zero}).
\end{proposition}

\begin{proof}
Let $R$ be an arbitrary inverse semiring. Any map $f\colon \Z_0 \to R$ must send the unit $1$ of $\Z_0$ to the unit of $R$, must preserve addition and inverses, and must send $0$ to $0$. This completely determines the map and we find
\begin{align*}
f(n) = \begin{cases}
1+1+\dots+1 \text{\ ($n$ times)}, & \text{for $n > 0$}, \\
-1-1-\dots-1 \text{\ ($-n$ times)}, & \text{for $n < 0$}, \\
0_1, & \text{for $n = \widehat{0}$}, \\
0, & \text{for $n = 0$,}
\end{cases}
\end{align*}
where $<$ refers to the usual order on $\Z$.
(Recall that in $\Z_0$, the element $\widehat{0}$ is the zero in the original ring $\Z$, while $0$ is the newly adjoined additive identity.)

Moreover, it is easy to see that the function so-defined is indeed a semiring homomorphism.
Hence, it is the unique homomorphism from $\Z_0$ to $R$.
\end{proof}

Note then that a commutative inverse semiring is the same thing as a commutative $\Z_0$-algebra. It is known that the free commutative $R$-algebra over a semiring $R$ on a set of variables is simply given by multivariate polynomials in those variables with coefficients in $R$. Thus, we arrive at the following result.

\begin{proposition}
The free (commutative) inverse semiring on $n$ generators is $\Z_0[x_1,\dots,x_n]$.
\end{proposition}

Composing this free functor with the reflections from \cref{sec:fundamental}, we expect that the ring reflection of the free inverse semiring on $n$ generators should be the free commutative ring $\Z[x_1,\dots,x_n]$ and the idempotent semiring of idempotents should be the free idempotent semiring $P_{\mathrm{fin}}(\N^n)$, where $P_{\mathrm{fin}}$ is the finite powerset (since the composition of adjoints is the adjoint of the composite). This can readily be verified explicitly --- the ring reflection simply identifies $0$ to $\overline{0}$ in all coefficients, while the idempotent associated to a polynomial corresponds to the set of exponents for the terms that have (truly) nonzero coefficients.
Similar results, of course, hold for the algebras over any inverse semiring, and in particular, for algebras over $\R_0$.

\begin{remark}
We can also recover the inverse semiring $\R[x]_B$ of polynomials from \cref{ex:polynomials}, which has a more impoverished semilattice of idempotents, by further quotienting $\R_0[x]$ by the relation $0_1 \le 0_x$ (i.e.\ $0_1 + 0_x = 0_x$). This destroys information about precisely which coefficients are used and only remembers the exponent of the largest power of $x$ -- that is, the degree.
Thus, this inverse semiring satisfies the universal property that for a (commutative) $\R_0$-algebra $A$ and any $a \in A$ such that $0_1 \le 0_a$, there is a unique map from $\R[x]_B$ to $A$ sending $x$ to $a$.

For multivariate polynomials, we may similarly quotient $\R_0[x_1,\dots,x_n]$ by $0_1 \le 0_{x_i}$ for $1 \le i \le n$ so that the idempotents can be interpreted as bounds on the multi-degrees of the polynomials. Further quotienting by $0_{x_i} = 0_{x_j}$, we obtain an inverse semiring whose idempotents correspond to the total degree.
\end{remark}

\begin{remark}\label{rem:zeroless_polys}
The situation is perhaps even clearer if we work with zeroless inverse semirings.
Since the left adjoint from zeroless inverse semirings to inverse semirings simply adjoins a zero, the free zeroless inverse semirings are as above, but with the zero removed.
For the zeroless version of bounded polynomials, we remove the zero $(0,-\infty)$, leaving $(0,0)$ as the new non-absorptive zero. These zeroless bounded polynomials are obtained from the free ones by just adding an equation that forces $0_1$ to be the smallest idempotent: $0_1 \le 0_a$ for all $a \in \Z_0[x_1,\dots,x_n] \setminus \{0\}$. Thus, they satisfy a slightly simpler universal property that does not make reference to the $x_i$'s explicitly: \emph{bounded polynomials are the free zeroless semirings for which $0_1$ is the smallest idempotent}.
\end{remark}

\begin{remark}
The adjunctions from \cref{sec:fundamental} can also be understood (in the commutative case) in terms of change-of-semiring functors for semiring algebras, which arise from the maps $\Z_0 \twoheadrightarrow \Z$ in the case of rings and $\Z_0 \twoheadrightarrow 2$ in the case of idempotent semirings. Recall that any homomorphism $f\colon R \to S$ between commutative semirings induces an adjunction between the categories of (commutative) $R$-algebras and $S$-algebras where the right adjoint views an $S$-algebra as an $R$-algebra by `restriction of scalars', while the left adjoint takes the tensor product of an $R$-algebra with $S$ over $R$.
Hence, the ring reflection corresponds to tensoring with the ring $\Z$ and the idempotent reflection from tensoring with the lattice $2$.
\end{remark}

\section{Modules and ideals}

In this section, we explore the theory of modules over an inverse semiring.
What should such a module be? Mimicking the approach for rings, we might say a module over an inverse semiring $R$ is a commutative inverse monoid $M$ equipped with a homomorphism from $R$ to $\End(M)$ (see \cref{endInvSR}). Alternatively, since inverse semiring is a special kind of semiring, we could use the notion of a semiring-module. The difference between these is that in the second case, the underlying monoid of $M$ is not assumed to be inverse. However, the following result shows that these two approaches actually agree.

Throughout this section, the terms `module' and `ideal' can be understood to mean `left module' and `left ideal', though the results for the right-sided versions are essentially the same.

\begin{proposition}\label{Mon=>InvMon}
Let $R$ be a general semiring and consider an $R$-module with underlying commutative monoid $M$. If $R$ is an inverse semiring, then $M$ is an inverse monoid.
\end{proposition}

\begin{proof}
Let $m\in M$. We want to prove that $m$ has an inverse. Consider $-1 \cdot m$. It is easy to see that
$m + -1 \cdot m + m = (1-1+1) \cdot m = 1 \cdot m = m$,
while
$-1\cdot m + m + -1\cdot m=(-1+1-1)\cdot m = -1 \cdot  m$.
Thus, $-1\cdot m$ is an inverse to $m$, and consequently, $M$ is an inverse monoid by \cref{rem:commuting_idemponents}.
\end{proof}

\begin{example}\label{Module->Ideal}
As is the case with rings, an inverse semiring $R$ is a module over itself and ideals of $R$ can be identified with submodules of this module.
\end{example}

\subsection{Subtractive submodules}

An important class of ideals in semirings is that of \emph{subtractive ideals}, with two-sided subtractive ideals corresponding to kernels of semiring homomorphisms.
The following definition is a generalisation from ideals (i.e.\ submodules of a semiring) to arbitrary submodules recalled here for the reader's convenience (see also \cite[Chapter 14]{Gol99}).

\begin{definition}
Let $R$ be a semiring.
A submodule $S$ of an $R$-module $M$ is called \emph{subtractive} if whenever $x \in S$, $y \in M$ and $x+y\in S$, we have that $y \in S$.

Note that for a ring $R$ every sub-$R$-module is subtractive.
\end{definition}

The aforementioned kernel result also holds for submodules. This is proved in \cite[Proposition 15.18]{Gol99}, but we prove it again here for completeness. This result involves homomorphisms of inverse semiring modules, which are defined as expected and agree with the usual notion of homomorphism for modules over a semiring. The following lemma makes the relation between ideals and congruences explicit (see \cite[Example 15.3]{Gol99} and compare the case of inverse semigroups \cite[Theorem 2.4.1]{lawson1998inverse}).
\begin{lemma}\label{lem:ideal_congruence}
Let $R$ be a semiring, $M$ an $R$-module, and $S$ a submodule of $M$. Then the congruence generated by $s \sim 0$ for $s \in S$ is given by \[x \sim y \iff \exists s,s' \in S.\ x + s = y + s'.\]
The quotient of $M$ by this congruence is denoted by $M/S$.
\end{lemma}
\begin{proof}
Let $\sim$ denote the congruence corresponding to the quotient $M \twoheadrightarrow M/S$ and the $\approx$ denote the relation defined by $x \approx y \iff \exists s,s' \in S.\ x + s = y + s'$. We show that they agree.
    
First note that if $x + s = y + s'$, then in $M/S$ we have $[x] = [x + 0] = [x + s] = [y + s'] = [y + 0] = [y]$, and hence ${\approx} \subseteq {\sim}$.
Also observe that we have $s \approx 0$ for all $s \in S$. Thus, it suffices to show that $\approx$ is a congruence.
    
It is clearly reflexive and symmetric. Closure under the operations is clear since $S$ is a submodule. To see that it is transitive, suppose $x + s = y + s'$ and $y + t = z + t'$ with $s,s',t,t' \in S$. Then \[x+s+t = y + s' + t = y + t + s' = z + t' + s'.\] Since $S$ is closed under addition, $s+t, s'+t' \in S$ and so the result follows.
\end{proof}

\begin{theorem}\label{Ker<=>SubtSubM}
Let $R$ be a semiring. For any $R$-module $M$, a submodule $S$ of $M$ is subtractive exactly when it arises as the kernel of some homomorphism $h\colon M \to N$ of $R$-modules --- that is, as the preimage of $0$ under $h$.
\end{theorem}
\begin{proof}
We prove the forward and reverse implications in turn. Firstly, assume that $h\colon M\to N$ is a homomorphism of the $R$-modules. Since we already know $\Ker(h)$ is a submodule of $M$, we only need to verify subtractivity. Assume that $a\in\Ker(h)$ and let $b\in M$ such that $a+b\in\Ker(h)$. Then
\[h(b)=0+h(b)=h(a)+h(b)=h(a+b)=0,\]
and so $b\in\Ker(h)$ and $\Ker(h)$ is subtractive.

For the other direction, let $S$ be a subtractive submodule of $M$. Consider the quotient map
$q\colon M \twoheadrightarrow M/S$. This quotient sends elements of $S$ to $0$ by construction. We must show that there are no other such elements.
If $[m]=0$ then by \cref{lem:ideal_congruence} we have $s,s'\in S$ such that $m+s=s'$.
Thus, $m\in S$ by subtractivity of $S$. This completes the proof.
\end{proof}

Our first result distinguishing subtractive modules in inverse semirings from subtractive submodules in general semirings relates them to a property of the canonical order from \cref{InvSR preorder}.

\begin{proposition}\label{DCMod<=>SubtrMod}
Let $R$ be an inverse semiring.
Any submodule $S$ of an $R$-module $M$ is subtractive if and only if it is downward closed with respect to the canonical order on (the underlying additive monoid of) $M$.
\end{proposition}

\begin{proof}
$(\Leftarrow)$ Let $S$ be downward closed and consider elements $a\in S$ and $b\in M$ such that $a+b\in S$.
Since $a\in S$ we know by \cref{Mon=>InvMon} that $-a=-1 \cdot a \in S$. Then $-a+a+b=0_a+b\in S$ and since $b\le 0_a+b$ by definition of the order, we know that $b\in S$. Thus, $S$ is subtractive.

$(\Rightarrow)$ Let $S$ be subtractive and consider an element $a\in S$.
Consider $b\in M$ such that $b\le a$. By \cref{prp:preorderadd}, we have $b+0_a=a$. But $0_a \in S$ and so $b \in S$ by subtractivity. Thus, $S$ is downward-closed, completing the proof. 
\end{proof}

\begin{corollary}\label{M/S=M/da(S)}
Let $R$ be an inverse semiring, $M$ an $R$-module, and $S$ a submodule of $M$.
Then the kernel of the quotient map $M \twoheadrightarrow M/S$ is given by $\da S$.
\end{corollary}

\begin{remark}
The quotient corresponding to the ring reflection of an inverse semiring $R$ from \cref{prop:reflections} can be described as $R/E(R)$ (since it identifies each $0_x$ with $0$).
Of course, this is the same as the quotient $R/\da E(R)$ and so the elements that get sent to $0$ are precisely the elements that compare less than some idempotent.
We will consider the case where $E(R)$ is already downwards closed in \cref{sec:concluding}.
\end{remark}

\subsection{Lattices of submodules}

It is a well-known result that the submodules of a module over a ring form a modular lattice --- that is, a lattice satisfying $(s_1 \vee x) \wedge s_2 = s_1 \vee (x \wedge s_2)$ whenever $s_1 \le s_2$. The same cannot be said for modules over more general semirings.
The issue already arises for ideals of commutative idempotent semirings, as can be seen in the following example from \cite[\S 12]{ward1939}.
\begin{example}\label{ex:nonmodular}
Let $\mathfrak{B}_1 = \{1, j, a, b, m, 0\}$ with the following order structure.
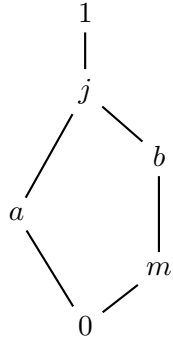
\begin{center}
\begin{tikzpicture}[auto]
  \matrix (nodes) [matrix of nodes, nodes in empty cells, ampersand replacement=\&, column sep=0.45cm, row sep=0.3cm]{
    \& $1$ \& \\[0.2cm]
    \& $j$ \& \\
    \& \& $b$ \\
    $a$ \& \& \\
    \& \& $m$ \\
    \& $0$ \& \\
  };
  \draw [thick] (nodes-1-2) -- (nodes-2-2);
  \draw [thick] (nodes-2-2) -- (nodes-3-3);
  \draw [thick] (nodes-2-2) -- (nodes-4-1);
  \draw [thick] (nodes-3-3) -- (nodes-5-3);
  \draw [thick] (nodes-4-1) -- (nodes-6-2);
  \draw [thick] (nodes-5-3) -- (nodes-6-2);
\end{tikzpicture}
\end{center}
Let addition in $\mathfrak{B}_1$ be given by join and multiplication be given as in the following table.
\begin{table}[H]
\centering
\begin{tabular}{c|cccccc}
$\boldsymbol{\cdot}$ & 1 & $j$ & $a$ & $b$ & $m$ & 0 \\
\hline
1 & 1 & $j$ & $a$ & $b$ & $m$ & 0 \\
$j$ & $j$ & $a$ & $a$ & 0 & 0 & 0 \\
$a$ & $a$ & $a$ & $a$ & 0 & 0 & 0 \\
$b$ & $b$ & 0 & 0 & 0 & 0 & 0 \\
$m$ & $m$ & 0 & 0 & 0 & 0 & 0 \\
0 & 0 & 0 & 0 & 0 & 0 & 0 \\
\end{tabular}
\end{table}

The modular law for ideals of $\mathfrak{B}_1$ fails when $s_1 = (a)$, $s_2 = (j)$ and $x = (m)$.
\end{example}

A similar failure of modularity applies to subgroups of a noncommutative group, but in this case the equality does hold when $s_1$ is restricted to be \emph{normal} subgroup (i.e.\ the kernel of homomorphism).
In the setting of semiring modules, a different relaxation holds, where it is instead $s_2$ that is restricted to be a kernel. In lattice-theoretic terms, subtractive submodules are \emph{right modular} elements in the lattice of submodules. Note that this result is true for general semirings.

\begin{theorem}
For any semiring $R$ and $R$-module $M$, the lattice of submodules of $M$ satisfies the following restricted modular law:
\[(S_1+X)\cap S_2=S_1+(X\cap S_2)\]
for all $X \le M$ and all $S_1 \le S_2 \le M$ where $S_2$ is \emph{subtractive}.
\end{theorem}

\begin{proof}
The $\ge$ direction holds in any lattice. We show that $(S_1+X)\cap S_2 \subseteq S_1+(X\cap S_2)$.
Let $k\in (S_1+X)\cap S_2$. Then we can write $k$ as $s_1+x$ for some $s_1\in S_1$ and $x\in X$. Since $s_1+x\in S_2$ and $s_1\in S_1 \subseteq S_2$, we have that $x\in S_2$ by subtractivity of $S_2$. Thus, $x \in X \cap S_2$, and so $k\in S_1 + (X\cap S_2)$.
\end{proof}

We might hope that this would imply the lattice of subtractive submodules is modular, but this is not necessarily the case, as subtractive submodules are not closed under joins in the lattice of all submodules.
The join in the lattice of subtractive submodules is given by taking the `subtractive closure' of the join in the lattice of submodules (see \cite{SA92}). By \cref{DCMod<=>SubtrMod} the subtractive closure for modules over an inverse semiring is simply the downward closure.

The subtractive ideals of \cref{ex:nonmodular} are precisely the principal downsets and hence fail to form a modular lattice. (Consider $s_1 = \da m$, $s_2 = \da b$ and $x = \da a$.)

\subsection{Upward-closed submodules}\label{sec:upwards}

So far we have studied downward-closed submodules. The upward-closed submodules are also worth consideration. In fact, the quotient of an $R$-module $M$ by a submodule $S$ is an abelian group if and only if $\da S$ is upward-closed. To prove this we will need the following result.

\begin{proposition}
Let $M$ be a module over an inverse semiring $R$. A submodule $S$ of $M$ is upward-closed if and only if $E(M) \subseteq S$.
\end{proposition}

\begin{proof}
Suppose that $S$ is upward-closed. Of course, we have that $0 \in S$. But by \cref{prp:upclosed}, every idempotent is greater than or equal to 0. Hence, $E(M) \subseteq S$.

Now suppose that $E(M) \subseteq S$ and consider $x \in S$, $y\in M$ such that $x \le y$. Then $y = x + 0_y$. But $0_y \in E(M)\subseteq S$ and so $y \in S$.
\end{proof}

\begin{proposition}
Let $R$ be an inverse semiring and $M$ an $R$-module. Then $M / E(M) \cong M/\da E(M)$ is an abelian group. Consequently, $M/S$ is an abelian group for any upward-closed submodule $S$ of $M$. 
\end{proposition}

\begin{proof}
Note that all idempotents belong to the equivalence class of $0$. Hence $[x - x] = 0$ and so $M/\da E(M)$ is a group.
Any upward-closed submodule $S$ of $M$ contains $E(M)$. So $M/S$ is a further quotient of $M/E(M)$ and thus also an abelian group.
\end{proof}

\begin{proposition}\label{Thm: CharUC}
Let $R$ be an inverse semiring, $M$ an $R$-module, and $S$ a subtractive submodule of $M$. The following are equivalent:
\begin{enumerate}
\item $S$ is upward-closed,
\item $E(M)\subseteq S$,
\item $M/S$ is an abelian group.
\end{enumerate}
\end{proposition}

\begin{proof}
We have already shown $(1)\Leftrightarrow (2)$ and $(2)\Rightarrow (3)$. We now prove $(3)\Rightarrow(2)$. If $M/S$ is a group, then it has only one idempotent $[0]$. But since homomorphisms preserve idempotents, this means that all idempotents in $M$ have to be mapped to $[0]$ and so $E(M)\subseteq \da S = S$. This completes the proof.
\end{proof}

The following specialisation of \cref{Thm: CharUC} is of particular interest since a ring is simply an inverse semiring where the underlying additive structure is a group.
\begin{corollary}
Let $I$ be a subtractive two-sided ideal of an inverse semiring $R$.
The following are equivalent:
\begin{enumerate}
\item $I$ is upward-closed,
\item $E(R)\subseteq I$,
\item $0_1 \in I$,
\item $R/I$ is a ring.
\end{enumerate}
\end{corollary}

\section{Connections to inverse semigroups}\label{sec:concluding}

In this paper we have worked out the basic theory of inverse semirings and discussed the extent to which the theory of modules over inverse semirings resembles the theory of modules over rings. We have neglected up to now to consider the ways in which the theory of inverse semirings is similar to the theory of inverse semigroups.

Various concepts from inverse semigroup theory are relevant to inverse semirings, but for this section we will only focus on one of them. It is not hard to see that the ideal of idempotents $E(R)$ of an inverse semiring $R$ is subtractive precisely when its additive inverse semigroup is \emph{E-unitary}.

\begin{definition}
An inverse semiring $R$ is \emph{E-unitary} if its underlying additive semigroup is E-unitary --- that is, if $x \in R$ and $z \in E(R)$, then $x + z$ is an idempotent if and only if $x$ is an idempotent.
\end{definition}

All rings, all idempotent semirings, and the inverse semirings of bounded polynomials are E-unitary. In addition, whenever an inverse monoid $X$ is E-unitary, the associated endomorphism inverse semiring $\End(X)$ will also be E-unitary.

An important theorem of inverse semigroups states that an inverse semigroup $S$ is E-unitary if and only if it can be embedded into a semidirect product of a group and a semilattice (see \cite[Theorem 7.1.5]{lawson1998inverse}). When $S$ is commutative, as in our case, the semidirect product reduces to a product. A similar result holds for inverse semirings.

 \begin{theorem}\label{thm:e_unitary_embedding}
Let $R$ be an inverse semiring. Then $R$ is E-unitary if and only if $R$ can be embedded in the product of a ring and an idempotent semiring. In particular, the ring can be taken to be the ring reflection $R/E(R)$ and the idempotent semiring to be $E(R)$.
 \end{theorem}

 \begin{proof}
We start with the easier direction. Assume $R$ is a sub-semiring of $S \times T$, where $S$ is a ring and $T$ is an idempotent semiring. It is clear that all the idempotent elements are of the form $(0,z)$. Now if $(x,e) + (0,z) = (0,z')$ in $R \subseteq S \times T$, it must be that $x = 0$ and so $(x,e)$ is an idempotent, as required.

Next assume that $R$ is $E$-unitary. Consider the function $\phi\colon R \to R/E(R) \times E(R)$ mapping $x \in R$ to $([x],0_x)$. It is evident that this is a homomorphism and so we need only verify that $\phi$ is an injection.
     
Suppose $\phi(x) = \phi(y)$. This means that $0_x = 0_y$ and $[x] = [y]$ in $R/E(R)$ from which we can deduce that there are idempotents $z$ and $z'$ such that $x+z = y + z'$ by \cref{lem:ideal_congruence}. Subtracting $y$ from both sides yields $(x-y) + z = 0_y + z'$. Since the right-hand side is an idempotent and $R$ is E-unitary, we find that $x-y$ is an idempotent. 
But since $0_x = 0_y$, the elements $x$ and $y$ lie in the same group $\G(0_x)$ (in which $0_x$ is the only idempotent). Thus, $x = y$, as required.
 \end{proof}

A generalisation of another result concerning inverse semigroups is discussed in a forthcoming paper \cite{faul2024}. There it is shown that just as Clifford semigroups may be viewed as functors from their lattice of idempotents into abelian groups, inverse semirings can always be viewed as lax monoidal functors from their idempotent semiring of idempotents into abelian groups with tensor product.

\bibliographystyle{abbrv}
\bibliography{bibliography}

\begin{thebibliography}{10}

\bibitem{AD21}
Y.~Ahmed and W.~A. Dudek.
\newblock Generalized multiplicative derivations in inverse semirings.
\newblock {\em Ufa Math. J.}, 13(1):110--18, 2021.

\bibitem{bhuniya2014clifford}
A.~Bhuniya.
\newblock The {C}lifford semiring congruences on an additive regular semiring.
\newblock {\em Discuss. Math. Gen. Algebra Appl.}, 34(2):143--153, 2014.

\bibitem{dadhwal2024study}
M.~Dadhwal and G.~Devi.
\newblock A study on derivations of inverse semirings with involution.
\newblock {\em Proyecciones}, 43(2):383--400, 2024.

\bibitem{dog2021centralizers}
S.~Dog, D.~M. Florence, R.~Murugesan, and P.~Namasivayam.
\newblock Centralizers of semiprime inverse semirings.
\newblock {\em Sarajevo J. Math.}, 17(2):167--173, 2021.

\bibitem{faul2024}
P.~F. Faul and G.~Manuell.
\newblock Clifford semigroups, inverse semirings and the {G}rothendieck
  construction.
\newblock (To appear).

\bibitem{Gol99}
J.~S. Golan.
\newblock {\em Semirings and their applications}.
\newblock Springer, Dordrecht, 1999.

\bibitem{Kar74}
P.~H. Karvellas.
\newblock Inversive semirings.
\newblock {\em J. Austral. Math. Soc.}, 18:277--288, 1974.

\bibitem{lawson1998inverse}
M.~V. Lawson.
\newblock {\em Inverse semigroups: the theory of partial symmetries}.
\newblock World Scientific, Singapore, 1998.

\bibitem{palakawong2023characterizations}
P.~Palakawong~na Ayutthaya, B.~Pibaljommee, and P.~Pinakano.
\newblock Characterizations of regular and $k$-regular semirings by pure ideals
  and pure $k$-ideals.
\newblock {\em Afr. Mat.}, 34(2):20, 2023.

\bibitem{sara2016centralizer}
S.~Sara, M.~Aslam, and M.~Javed.
\newblock On centralizer of semiprime inverse semiring.
\newblock {\em Discuss. Math. Gen. Algebra Appl.}, 36(1):71--84, 2016.

\bibitem{SA92}
M.~K. Sen and M.~R. Adhikari.
\newblock On $k$-ideals of semirings.
\newblock {\em Int. J. Math. Math. Sci.}, 15(2):347--350, 1992.

\bibitem{SM04}
M.~K. Sen and S.~K. Maity.
\newblock A note on orthodox additive inverse semirings.
\newblock {\em Acta Univ. Palack. Olomuc. Fac. Rerum Natur. Math.},
  43:149--154, 2004.

\bibitem{sen2005clifford}
M.~K. Sen, S.~K. Maity, and K.~P. Shum.
\newblock Clifford semirings and generalized clifford semirings.
\newblock {\em Taiwanese J. Math.}, pages 433--444, 2005.

\bibitem{ward1939}
M.~Ward and R.~P. Dilworth.
\newblock Residuated lattices.
\newblock {\em Trans. Amer. Math. Soc.}, 45:335--354, 1939.

\bibitem{yaqoub2021generalized}
A.~Yaqoub et~al.
\newblock Generalized multiplicative derivatives in inverse semirings.
\newblock {\em Ufa Math. J.}, 13(1):110--118, 2021.

\bibitem{yilmaz2023note}
D.~Yilmaz.
\newblock A note on inverse semirings with bi-derivations.
\newblock {\em Collect. Math.}, 74(2):473--485, 2023.

\end{thebibliography}
\end{document}